\documentclass[11pt,a4paper]{article}

\usepackage{color}
\usepackage{graphics,epsfig}
\usepackage{amsfonts,amssymb,amsmath,amsthm}

\newtheorem{lem}{Lemma}[section]
\newtheorem{cor}[lem]{Corollary}
\newtheorem{thm}[lem]{Theorem}
\newtheorem{assumption}[lem]{Assumption}

\theoremstyle{definition}

\theoremstyle{remark}
\newtheorem{rem}[lem]{Remark}

\numberwithin{equation}{section}

\newcommand{\ep}{\varepsilon}
\newcommand{\ue}{u^\ep}
\newcommand{\n}{\nabla }
\newcommand{\eun}{\displaystyle{\frac{1}{\ep}}}
\newcommand{\R}{\mathbb{R}}
\newcommand{\vsp}{\vspace{8pt}}
\newcommand{\di}{\displaystyle}
\newcommand{\Pe}{(P^{\;\!\ep})}
\newcommand{\Pz}{(P^{\;\!\cl})}

\newcommand{\pt}{\ep|\ln\ep|m_1e^{m_2t}}
\newcommand{\cl}{c_\lambda}

\title{Sharp interface limit of the Fisher-KPP equation
when initial data have slow exponential decay}
\author{ }
\date{}

\begin{document}

\maketitle \vspace{-20 mm}

\begin{center}

{\large\bf Matthieu Alfaro }\\[1ex]
I3M, Universit\'e de Montpellier 2,\\
CC051, Place Eug\`ene Bataillon, 34095 Montpellier Cedex 5, France,\\[2ex]

{\large\bf Arnaud Ducrot }\\[1ex]
UMR CNRS 5251 IMB and INRIA Sud-Ouest ANUBIS, \\
Universit\'e de Bordeaux, 3, Place de la Victoire, 33000 Bordeaux France. \\[2ex]

\end{center}

\vspace{15pt}


\begin{abstract}

We investigate the singular limit, as $\ep \to 0$, of the Fisher
equation $\partial _t u=\ep \Delta u + \ep ^{-1}u(1-u)$ in the
whole space. We consider initial data with compact support plus
perturbations with {\it slow exponential decay}. We prove that the
sharp interface limit moves by a constant speed, which
dramatically depends on the tails of the initial data. By
performing a fine analysis of both the generation and motion of
interface, we provide a new estimate of the
thickness of the transition layers.\\

\end{abstract}

\section{Introduction}\label{s:intro}

In this work, we consider $\ue=\ue(t,x)$ the solution of the
rescaled Fisher-KPP equation
\[
 \Pe \quad\begin{cases}
 \partial _t \ue= \ep \Delta \ue+\eun \ue(1-\ue)&\text{in }(0,\infty)\times \R ^N  \vspace{3pt}\\
 \ue(0,x)=u_{0,\ep}(x)  &\text{in }\R ^N\,,
 \end{cases}
\]
with $\ep >0$ a small parameter, related to the thickness of a
diffuse interfacial layer. Let us recall that, in the classical
works of Fisher \cite{Fish} and Kolmogorov, Petrovsky and Piskunov
\cite{Kol-Pet-Pis}, the authors consider a smooth monostable
nonlinearity $f:[0,1]\to\R$ such that
$$
f(0)=f(1)=0\,,\quad 0<f(u)\leq f'(0)u\, \text{ for all }
u\in(0,1)\,.
$$
Our results would hold for such nonlinearities but, for the sake
of clarity, we restrict ourselves to the case where $f(u)=u(1-u)$.

In \cite{A-Duc} we have investigated the singular limit, as $\ep
\to 0$,  of $\Pe$ when initial data have compact support plus,
possibly, perturbations with a {\it fast exponential decay}. We
proved that the sharp interface limit moves by constant speed
which is the minimal speed $c^*=2$ of the underlying travelling
waves. We also obtained a new $\mathcal O (\ep |\ln \ep|)$
estimate for the thickness of the transition layers of the
solutions $\ue$.

The present paper is a completion of \cite{A-Duc}: we are
concerned with initial data with a {\it slow exponential decay}.
In this case, it turns out that the limit interface moves by a
speed which dramatically depends on the tails of the initial data.
We prove the convergence and again obtain a new $\mathcal O (\ep
|\ln \ep|)$ estimate for the thickness of the transition layers of
the solutions $\ue$.

\vskip 8pt We shall assume the following properties on the initial
data.
\begin{assumption}\label{H1}
We assume that $u_{0,\ep}=g+h_\ep$ where
\begin{itemize}
\item [(i)] $g$ is a bounded, nonnegative and compactly supported
function. We define $ \Omega_0:={\rm supp}\; g$.

\item [(ii)] We define $\widetilde g$ as the restriction of $g$ on
$\overline {\Omega _0}$ and assume that $\widetilde g$ is of the
class $C^2$.

\item [(iii)] $h_\ep$ is a  nonnegative function and there exist
$\lambda\in (0,1)$ and $0<m<M$ such that, for all $\ep >0$ small
enough,
\begin{equation*}
m e^{- \lambda\frac{|d(0,x)|}{\ep}}\leq h_\ep(x)\leq M e^{-
\lambda\frac{|d(0,x)|}{\ep}}\,,\;\;\forall x\in\R^N\,,
\end{equation*}
where $d(0,\cdot)$ denotes the cut-off signed distance function to
the ``initial interface" $\Gamma _0:=\partial \Omega_0$ (see
subsection \ref{ss:distance}).
\end{itemize}
\end{assumption}

\begin{rem}[Fast/Slow exponential decay]
In the fast exponential decay case considered in \cite{A-Duc},
$(iii)$ is replaced by $h_\ep(x)\leq M e^{-
\lambda\frac{\|x\|}{\ep}}$ for some $\lambda \geq 1$. In some
sense, see subsection \ref{ss:tw} for details, the exponential
decay of the initial data is faster than the exponential decay of
the underlying travelling wave of minimal speed $c^*$. In this
case, $\lambda \geq 1$ does not affect the asymptotic speed of the
limit interface which is always $c^*$.

In the present case $\lambda \in (0,1)$ we consider,
\eqref{plus-infini-U} indicates that the exponential decay of the
initial data and that of the underlying travelling wave of speed
$\lambda + \lambda ^{-1}$ are the same; it follows that the
construction of efficient sub-solutions is more involved than in
\cite{A-Duc}. Here, $\lambda \in(0,1)$ does affect the asymptotic
speed of the limit interface, which turns out to be $\lambda +
\lambda ^{-1}$.
\end{rem}

Note that the regularity assumption $(ii)$ can be relaxed. We
refer to Remark 1.8. in \cite{A-Duc}.

\begin{assumption}\label{H2}
We assume that $\Omega_0$ is convex.
\end{assumption}

\begin{assumption}\label{H3}
We assume the existence of $\delta >0$ such that, if $n$ denotes
the Euclidian unit normal vector exterior to the initial interface
$\Gamma _0$, then
\begin{equation}\label{pente}
\left| \frac{\partial \widetilde g}{\partial n}(y)\right|\geq
\delta \quad\text{ for all } y\in \Gamma _0\,.
\end{equation}
\end{assumption}
Assumption \ref{H1} gives the structure of our allowed initial
data. Note that Assumption \ref{H2} is used to find upper bounds
for the solutions $\ue$ (see Lemma \ref{lem:sup}) whereas
Assumption \ref{H3} is only used to derive the correspondence
\eqref{pente-bis}.

\vskip 8 pt Before going into much details, let us comment on
related known results. It is long known that the tails of the
initial data play a key role in the study of the long time
behavior of $u=u(t,x)$ the solution of the Fisher-KPP equation
$\partial _t u=\Delta u +u(1-u)$. As far as initial data with
exponential decay are concerned, we refer among others to
\cite{McK}, \cite{Bra}, \cite{Lau} for a probabilistic framework
and to \cite{Lar}, \cite{Rot}, \cite{Uch} for a reaction-diffusion
framework. More recently, Hamel and Roques \cite{Ham-Roq} studied
the case where the initial data decays more slowly than any
exponentially decaying function.

In a singular limit framework, the question of the convergence of
Problem $\Pe$ has been addressed when the initial data $u_{0,\ep}$
does not depend on $\ep$ and is compactly supported : first by
Freidlin \cite{Frie} using probabilistic methods and later by
Evans and Souganidis \cite{Ev-Soug} using Hamilton Jacobi technics
(in this framework we also refer to \cite{Bar-Eva-Sou,
Bar-Soug2}). In \cite{A-Duc}, we provide a new proof of
convergence for Problem $\Pe$ with fast exponentially decaying
initial data, by using specific reaction-diffusion tools such as
the comparison principle. Moreover, we obtain an $\mathcal O
(\ep|\ln \ep|)$ estimate of the thickness of the transition layers
of the solutions $\ue$. To the best of our knowledge, no such fine
estimate of the thickness of the transition layers existed for the
Fisher-KPP equation (in contrast with the Allen-Cahn equation).

\vskip 8pt  As $\ep \rightarrow 0$, by formally neglecting the
diffusion term, we see that, in the very early stage, the value of
$\ue$ quickly becomes close to either $1$ or $0$ in most part of
$\R^N$, creating a steep interface (transition layer) between the
regions $\{\ue\approx 1\}$ and $\{\ue\approx 0\}$. Once such an
interface develops, the diffusion term is large near the interface
comes to balance with the reaction term. As a result, the
interface ceases rapid development and starts to propagate in a
much slower time scale. Therefore the limit solution $\tilde u
(t,x)$ will be a step function taking the value $1$ on one side of
the moving interface, and $0$ on the other side.

We shall prove that this sharp interface, which we will denote by
$\Gamma ^{\cl}_t$, obeys the law of motion
\[
 \Pz\quad\begin{cases}
 \, V_{n}=\cl:=\lambda+\lambda^{-1}
 \quad \text { on } \Gamma ^{\cl}_t \vspace{3pt}\\
 \, \Gamma_t^{\cl}\big|_{t=0}=\Gamma_0\,,
\end{cases}
\]
where $V_n$ denotes the normal velocity of $\Gamma ^{\cl} _t$ in
the exterior direction. Note that
$\cl=\lambda+\lambda^{-1}>2=c^*$, with $c^*=2$ the minimal speed
of some related one-dimensional travelling waves (see subsection
\ref{ss:tw} for details). Therefore, as expected, the slower the
initial data decay, the larger is the the speed of the sharp
interface limit.

Since the region enclosed by $\Gamma _0$, namely $\Omega _0$, is
smooth and convex, Problem $\Pz$, possesses a unique smooth
solution on $[0,\infty)$, which we denote by $\Gamma
^{\cl}=\bigcup _{t\geq 0} (\{t\}\times\Gamma^{\cl}_t)$. Hereafter,
we fix $T>0$ and work on $(0,T]$.

For each $t\in (0,T]$, we denote by $\Omega ^{\cl} _ t$  the
region enclosed by the hypersurface $\Gamma ^{\cl} _t$. We define
a step function $\tilde u (t,x)$ by
\begin{equation}\label{u}
\tilde u  (t,x)=\begin{cases}
\, 1 &\text{ in } \Omega ^{\cl} _t\\
\, 0 &\text{ in } \R ^N \setminus \overline{\Omega ^{\cl} _t}
\end{cases} \quad\text{for } t\in(0,T]\,,
\end{equation}
which represents the asymptotic limit of $\ue$  as $\ep\to 0$.

\vskip 8 pt Our main result, Theorem \ref{th:results}, shows that,
after a short time of order $\mathcal O(\ep|\ln \ep|)$, the
solution $\ue$ quickly becomes close to 1 or 0, except in a small
neighborhood of the initial interface $\Gamma _0$, creating a
steep transition layer around $\Gamma _0$ ({\it generation of
interface}). The theorem then states that the solution $\ue$
remains close to the step function $\tilde u$ on the time interval
$[t^\ep,T]$ ({\it motion of interface}). Last, \eqref{resultat}
shows that, for any $0<a<1$, for all $t^\ep\leq t \leq T$, the
level-set $\Gamma _t ^\ep (a):=\{x\in\R ^N:\, \ue(t,x)=a\}$ lives
in an $\mathcal O (\ep|\ln \ep|)$ tubular neighborhood of the
limit interface $\Gamma _t ^{\cl}$. In other words, we provide a
new estimate of the {\it thickness of the transition layers} of
the solutions $\ue$.

\begin{thm}[Generation, motion and thickness of interface]\label{th:results}
Let Assumptions \ref{H1}, \ref{H2} and \ref{H3} be satisfied. Then
there exist positive constants $\alpha$ and $\mathcal C$ such
that, for all $\ep >0$ small enough, for all $t^\ep \leq t \leq
T$, where
$$
t^\ep:=\alpha \ep  |\ln \ep|\,,
$$
we have
\begin{equation}\label{resultat}
\ue(t,x) \in
\begin{cases}
\,[0,1+\ep]\quad&\text{if}\quad
x\in\mathcal N_{\mathcal C\ep|\ln\ep|}(\Gamma ^{\cl}_t)\\
\,[1-2\ep,1+\ep]\quad&\text{if}\quad x\in\Omega ^{\cl} _
t\setminus\mathcal N_{\mathcal C\ep|\ln\ep|}(\Gamma ^{\cl}
_t)\\
\,[0,\ep]\quad&\text{if}\quad x\in (\R^N\setminus \overline{\Omega
^{\cl}_t})\setminus\mathcal N_{\mathcal C\ep|\ln\ep|}(\Gamma
^{\cl} _t)\,,
\end{cases}
\end{equation}
where $\mathcal N _r(\Gamma ^{\cl} _t):=\{x\in \R ^N:\,
dist(x,\Gamma ^{\cl} _t)<r\}$ denotes the tubular $r$-neighborhood
of $\Gamma ^{\cl}_t$.
\end{thm}

As a immediate consequence of the above Theorem, we collect the
convergence result.

\begin{cor}[Convergence]\label{cor:cv}Let Assumptions \ref{H1}, \ref{H2}
and \ref{H3} be satisfied. Then, as $\ep\to 0$, $\ue$ converges to
$\tilde u $ everywhere in $\bigcup _{0< t\leq T}\left(\{ t\}\times
\Omega ^{\cl} _t\right)$ and $\bigcup _{0< t \leq T}\left(\{
t\}\times (\R^N \setminus \overline{\Omega ^{\cl} _t})\right)$.
\end{cor}

\vskip 8 pt The organization of this paper is as follows. In
Section \ref{s:materials}, we present the basic tools that will be
used in later sections for the construction of sub- and
super-solutions. In Section \ref{s:sub} we construct two
sub-solutions, one for small times (during the generation of
interface) and one for later times (during the motion of
interface).  Section \ref{s:super} is devoted to the construction
of a single super-solution which is efficient to study both the
generation and the motion of interface. Last, in Section
\ref{s:proof}, by using our different sub- and super-solutions we
prove Theorem \ref{th:results}.

\section{Materials}\label{s:materials}

The needed tools are the same than in \cite{A-Duc}. For the
self-containedness of the present paper we recall them here. Let
us note than the non monotone travelling waves of speed $0<c<c^*$
used in \cite{A-Duc} are useless here.

\subsection{A monostable ODE}\label{ss:ode}

The generation of interface is strongly related to the dynamical
properties of the ordinary differential equation associated to
$\Pe$, that is
\begin{equation*}
\frac{dz(t)}{dt}=z(t)(1-z(t))\,,\;\;t>0\,.
\end{equation*}
In the sequel, for technical reasons we shall apply the semiflow
generated by the above dynamical system to negative initial data.
In order to have some good dynamical properties, let us modify the
monostable nonlinearity $u\to u(1-u)$  on $(-\infty,0)$ so that
the modified function, we call it $\bar f:\R\to\R$, is of the
class $C^2$ and enjoys the bistable assumptions. More precisely,
$\bar f$ has exactly three zeros $-1<0<1$ and
\begin{equation}
{\bar f}'(-1)<0\,, \qquad \bar f'(0)=1>0\,, \qquad \bar
f'(1)=-1<0\,.
\end{equation}
Note that $\bar f(u)=u(1-u)$ if $u\geq 0$. As done in Chen
\cite{C1}, we consider $\bar f _\ep$ a slight modification of
$\bar f$ defined by
$$
\bar f
_\ep(u):=\psi(u)\frac{u-\ep|\ln\ep|}{|\ln\ep|}+(1-\psi(u))\bar
f(u)\,,
$$
with $\psi$ a smooth cut-off function satisfying conditions
(29)---(32) as they appear in \cite{HKLM}. As explained in
\cite{HKLM},
\begin{equation}\label{fep}
\bar f _ \ep (u) \leq \bar f(u)\quad \text{ for all }u\in\R\,.
\end{equation}

Then we defined $w(s,\xi)$ as the semiflow generated by the
ordinary differential equation
\begin{equation}\label{ode}
\begin{cases}
\di{\frac{dw}{ds}}(s,\xi)=\bar f_ \ep (w(s,\xi))\,,\;\;s>0\,,\vsp\\
w(0,\xi)=\xi\,.
\end{cases}
\end{equation}
Here $\xi$ ranges over the interval
$[-\|g\|_{\infty}-M-1,\|g\|_{\infty}+M+1]$. We claim that
$w(s,\xi)$ has the following properties (for proofs, see \cite{C1}
or \cite{HKLM}).

\begin{lem}[Behavior of $w$]\label{lem:w}The following
holds for all $\xi \in[-\|g\|_{\infty}-M-1,\|g\|_{\infty}+M+1]$.
\begin{itemize}
\item [(i)] $\text{ If}\quad \xi \geq \ep|\ln \ep| \quad\text{ then}\quad w(s,\xi)\geq\ep|\ln\ep|>0 \quad\text{ for all } s>0\,.$\\
$\text{ If}\quad \xi <0 \quad\text{ then}\quad w(s,\xi)<0
\quad\text{ for all } s>0\,.$\\
$\text{ If}\quad \xi \in(0,\ep|\ln\ep|) \quad\text{ then}\quad
w(s,\xi)>0 \quad\text{ for all } s\in(0,s_\ep(\xi)), \text{ with
}$
$$
s_\ep(\xi):=|\ln\ep|\left|\ln\left(1-\frac{\xi}{\ep|\ln\ep}\right)\right|\,.
$$
 \item [(ii)] $w(s,\xi) \in
(-\|g\|_{\infty}-M-1,\|g\|_{\infty}+M+1)\quad\text{ for all }
s>0\,.$ \item [(iii)] $w$ is of the class $C^2$ with respect to
$\xi$ and $$ w_\xi (s,\xi)>0\quad\text{ for all } s>0\,.$$

 \item [(iv)]For all $a>0$, there exists a constant $C(a)$ such that
$$
 \left|\di{\frac{w_{\xi\xi}}{w_\xi}(s,\xi)}\right|\leq \frac {C(a)} \ep \quad\text{ for all }
 0<s\leq a|\ln \ep|\,.
 $$
\item [(v)] There exists a positive constant $\alpha$ such that,
for all $s\geq \alpha |\ln \ep|$, we have
$$
\text{ if }\quad \xi \in[\ep|\ln\ep|,\|g\|_{\infty}+M+1]
\quad\text{ then }\quad 0<w(s,\xi)\leq 1+\ep\,,
$$
and
$$
\text{ if }\quad \xi \in[3\ep|\ln\ep|,\|g\|_{\infty}+M+1]
\quad\text{ then }\quad 1-\ep \leq w(s,\xi)\,.
$$
\end{itemize}
\end{lem}

\subsection{Travelling waves}\label{ss:tw}

\vsp A travelling wave is a couple $(c,U)$ with $c>0$ and $U\in
C^2(\R,\R)$ a function such that
\begin{equation}\label{tw}
\begin{cases}
 {U}''(z)+cU'(z)+U(z)(1-U(z))=0\quad \text{ for all } z\in \R\\
 U(-\infty)=1\\
 U(\infty)=0\,.
\end{cases}
\end{equation}
Define $c^*:=2$. Then, for all $c\geq c^*$ there exists a unique
(up to a translation in $z$) travelling wave denoted by $(c,U)$.
It is positive and monotone.

\begin{lem}[Behavior of $U$]\label{lem:beh-U}Let $c>2=c^*$ be arbitrary and consider the associated
travelling wave $(c,U)$. Then there exist constants $C>0$ and
$0<r<R$ such that
\begin{equation}\label{moins-infini}
r e^{-\eta |z|} \leq 1-U(z)\leq R e^{-\eta |z|} \quad \text{ for }
z\leq 0\,,
\end{equation}
\begin{equation}\label{plus-infini}
r e^{-\mu  z}\leq U(z)\leq R e^{-\mu z} \quad \text{ for } z\geq
0\,,
\end{equation}
\begin{equation}\label{moins-infini-der}
re^{-\eta |z|}\leq |U'(z)|+|U''(z)|\leq Re^{-\eta |z|} \quad
\text{ for } z\leq 0\,,
\end{equation}
\begin{equation}\label{plus-infini-der}
re^{-\mu |z|}\leq |U'(z)|+|U''(z)|\leq Re^{-\mu |z|} \quad \text{
for } z\geq 0\,,
\end{equation}
with $\eta>0$ the positive root of equation $\eta ^2+c\eta -1=0$
and $\mu>0$ the smallest root of equation $\mu ^2-c\mu +1=0$.
\end{lem}

We refer the reader to \cite{Aro-Wei1,Volpert-Volpert-Volpert} and
the references therein for more details.

\subsection{Cut-off signed distance functions}\label{ss:distance}

Recall that $\Gamma ^{\cl}=\bigcup _{t\geq 0}
(\{t\}\times\Gamma^{\cl}_t)$ is the smooth solution of the free
boundary problem $(P^{\cl})$ and that, for each $t>0$, $\Omega
^{\cl} _ t$  is the region enclosed by the hypersurface $\Gamma
^{\cl} _t$.

Let $\widetilde d (t,\cdot)$ be the signed distance function to
$\Gamma ^{\cl} _t$ defined by
\begin{equation}\label{eq:dist}
\widetilde d (t,x)=
\begin{cases}
-&\hspace{-10pt}\mbox{dist}(x,\Gamma ^{\cl} _t)\quad\text{ for }x\in\Omega ^{\cl}  _t \\
&\hspace{-10pt} \mbox{dist}(x,\Gamma ^{\cl} _t) \quad \text{ for }
x\in\R ^N \setminus \Omega ^{\cl} _t\,.
\end{cases}
\end{equation}
We remark that $\widetilde d=0$ on $\Gamma ^{\cl}$ and that
$|\nabla \widetilde d|=1$ in a neighborhood of $\Gamma ^{\cl}$.

We now introduce the ``cut-off signed distance function" $d$,
which is defined as follows. Recall that $T>0$ is fixed. First,
choose $d_0>0$ small enough so that $\widetilde d$ is smooth in
the tubular neighborhood of $\Gamma ^{\cl}$
\[
 \{(t,x) \in [0,T] \times \R ^N:\;\;|\widetilde{d}(t,x)|<4
 d_0\}\,.
\]
Next let $\zeta(s)$ be a smooth function satisfying
\begin{equation}\label{zeta}
0\leq\zeta'(s)\leq 1\quad\text{ for all } s\in\R\,,
\end{equation}
such that
\[
 \zeta(s)= \left\{\begin{array}{ll}
 s &\textrm{ if }\ |s| \leq d_0\vspace{4pt}\\
 -2d_0 &\textrm{ if } \ s \leq -3d_0\vspace{4pt}\\
 2d_0 &\textrm{ if } \ s \geq 3d_0\,.
 \end{array}\right.
\]
We then define the cut-off signed distance function $d$ by
\begin{equation}
d(t,x):=\zeta\left(\tilde{d}(t,x)\right)\,.
\end{equation}

Note that
\begin{equation}\label{norme-un}
\text{ if } \quad |d(t,x)|< d_0 \quad \text{ then }\quad |\nabla
d(t,x)|=1\,,
\end{equation}
and that the equation of motion $(P^{\cl})$ yields
\begin{equation}\label{interface}
\text{ if } \quad |d(t,x)|< d_0 \quad \text{ then }\quad
 \partial _t d(t,x)+\cl=0\,.
\end{equation}
Then the mean value theorem provides a constant $A>0$ such that
\begin{equation}\label{MVT}
|\partial _t d(t,x)+\cl|\leq A|d(t,x)| \quad \textrm{ for all }
(t,x) \in [0,T]\times \R ^N\,.
\end{equation}
Moreover, there exists a constant $C>0$ such that
\begin{equation}\label{est-dist}
|\nabla d (t,x)|+|\Delta d (t,x)|\leq C\quad \textrm{ for all }
(t,x) \in [0,T]\times \R ^N\,.
\end{equation}

\section{Sub-solutions}\label{s:sub}

As explained before, the construction of efficient sub-solutions
is more involved than in the fast exponential decay case
considered in \cite{A-Duc}. We start by constructing refined
sub-solutions for small times.

\subsection{Generation of interface}\label{ss:generation}

Let us recall that $m e^{- \lambda\frac{|d(0,x)|}{\ep}}\leq
h_\ep(x)\leq M e^{- \lambda\frac{|d(0,x)|}{\ep}}$ for all $
x\in\R^N$. We then define the map
\begin{equation}\label{sub-sol}
\underline{u}(t,x):=\max\left\{\tilde m e^{- \lambda
\frac{|d(0,x)|}{\ep}}, w\left(\frac{t}{\ep}\,,
g(x)-Kt\right)\right\},
\end{equation}
where $w(s,\xi)$ is the solution of the ordinary differential
equation \eqref{ode} and where $\tilde m=\min(\frac 1 2 ,m)$. The
constant $K>0$ is  to be specified below.

\begin{lem}[Sub-solutions for small times]\label{lem:sub-small}
Let Assumption \ref{H1} be satisfied. Then for all $a>0$, there
exists $K>0$ such that, for all $\ep >0$ small enough, we have
\begin{equation}\label{sub-sol2}
\underline{u}(t,x)\leq u^\ep(t,x)\,,\;\;\forall t\in [0,a \ep
|\ln\ep|]\,,\; \forall x\in \mathbb R^N\,.
\end{equation}
\end{lem}

\begin{proof}
Let us first notice that, for all $x\in\R ^N$,
\begin{equation*}
\underline{u}(0,x)=\max\left(\tilde m e^{-
\lambda\frac{|d(0,x)|}{\ep}},g(x)\right)\leq m e^{-
\lambda\frac{|d(0,x)|}{\ep}}+g(x)= \ue (0,x)\,.
\end{equation*}
Then it remains to show that $\underline{u}$ is a sub-solution of
Problem $\Pe$.

Let us consider the operator
\begin{equation*}
\mathcal L^\ep [v]:=\partial _t v-\ep\Delta
v-\frac{1}{\ep}v(1-v)\,.
\end{equation*}
Let $a>0$ be arbitrary. We show below that, if $K>0$ is
sufficiently large then, for all $\ep >0$ small enough, $\mathcal
L^\ep [\underline{u}]\leq 0$. We distinguish two cases.

We first consider the points $(t,x)$ where $\underline u
(t,x)=\tilde m e^{- \lambda \frac{|d(0,x)|}{\ep}}$. Assume further
that $d(0,x)\geq 0$ so that $\underline u (t,x)=\tilde m e^{-
\lambda \frac{d(0,x)}{\ep}}=:\varphi _\ep(x)$. We compute
$$
\begin{array}{ll}
\partial_t \underline{u}=0 \vsp\\
\Delta \underline{u}= \lambda ^2 \di{\frac {|\n d|^2}{\ep ^2}}
\varphi _\ep-\lambda \di{\frac{\Delta d}{\ep}} \varphi _\ep\,.
\end{array}
$$
Therefore, we get
$$
\ep \mathcal L ^\ep[\underline u](t,x)= \varphi _\ep\left(-\lambda
^2|\n d|^2+\ep \lambda  \Delta d -1+\varphi _\ep\right)\,.
$$
Since $0\leq \varphi _\ep \leq 1/2$, we get
$$
\ep \mathcal L ^\ep[\underline u](t,x)\leq \varphi _\ep\left(\ep
\lambda \Delta d -\frac 12\right)\leq 0\,,
$$
for $\ep >0$ sufficiently small. The case where $d(0,x)\leq 0$ is
very similar and omitted.

Next we consider the points $(t,x)$ where $\underline u
(t,x)=w\left(\frac{t}{\ep}\,, g(x)-Kt\right)$. We have
$$
\begin{array}{ll}
\partial_t \underline{u}=\eun w_s-K w_\xi\vsp\\
\Delta \underline{u}= w_{\xi\xi} |\nabla g|^2+w_\xi \Delta g\,.
\end{array}
$$
Then, we get
\begin{equation*}
\begin{array}{lll}
\mathcal L^\ep [\underline{u}](t,x)&=\di \frac{1}{\ep} w_s-K
w_\xi-\ep\left( w_{\xi\xi} |\nabla g|^2+w_\xi
\Delta g\right)-\frac{1}{\ep} w(1-w)\vsp\\
&\leq\di \frac 1 \ep w_s -K
w_\xi-\ep\left( w_{\xi\xi} |\nabla g|^2+w_\xi \Delta g\right)-\frac 1 \ep \bar f _\ep (w)\vsp\\
&=-w_\xi\left [K+\ep\left( \frac{w_{\xi\xi}}{w_\xi} |\nabla g|^2+
\Delta g\right)\right]\,,
\end{array}
\end{equation*}
where we have successively used \eqref{fep} and \eqref{ode}. In
view of Lemma \ref{lem:w} $(iv)$, there exists a constant $C(a)>0$
such that, for all $(t,x)$ with $0\leq t \leq a\ep|\ln\ep|$, we
have
\begin{equation*}
\left\lvert\frac{w_{\xi\xi}}{w_\xi} |\nabla g|^2+ \Delta
g\right\rvert \leq \frac{C(a)}{\ep}\,.
\end{equation*}
Therefore, choosing $K>C(a)$ implies
\begin{equation*}
\mathcal L^\ep [\underline{u}](t,x)\leq -w_\xi
\left(K-C(a)\right)\leq 0\,,
\end{equation*}
since $w_\xi >0$.

This completes the proof of Lemma \ref{lem:sub-small}.
\end{proof}

As a consequence of the above construction of sub-solutions for
small times, we deduce that, after a very short time, the solution
$u^\ep$ approaches 1 in most part of the support of the initial
data.

\begin{cor}[Generation of interface \lq\lq from the inside"]\label{cor:gen-inside}
Let Assumption \ref{H1} be satisfied. Then there exist $k>0$,
$\alpha>0$ such that, for all $\ep>0$ small enough,
\begin{equation}\label{def:k}
d(0,x)\leq -k \ep|\ln \ep|\, \Longrightarrow\, 1-\ep\leq
\ue(t^\ep,x)\leq 1+\ep\,,
\end{equation}
wherein $t^\ep:=\alpha\ep |\ln\ep|$.
\end{cor}

\begin{proof}In Lemma \ref{lem:sub-small}, we select
$a=\alpha$ where $\alpha >0$ is as in Lemma \ref{lem:w} $(v)$.
Note that, in view of \eqref{pente}, the mean value theorem
provides the existence of a constant $k>0$ such that
\begin{equation}\label{pente-bis}
\text{ if } \quad d(0,x)\leq -k \ep|\ln\ep| \quad \text{ then }
\quad g(x)\geq (3+K\alpha)\ep|\ln\ep|\,.
\end{equation}
Then, using \eqref{sub-sol2}, we see that, for all $x$ satisfying
$d(0,x)\leq -k\ep|\ln\ep|$,
$$
\ue(t^\ep,x)\geq \underline u (t^\ep,x)\geq w(\alpha|\ln\ep|,3
\ep|\ln\ep|)\geq 1-\ep\,.
$$
The last inequality follows from Lemma \ref{lem:w} $(v)$.

Note that the upper bound $1+\ep$ is actually valid for all $x$
and all $t\geq t^\ep$, as seen in \eqref{recentrage}, and will be
proved in Section \ref{s:proof}.

The corollary is proved.
\end{proof}

\subsection{Motion of interface}\label{ss:motion}

In the fast exponential decay case, the limit speed of the sharp
interface limit is $c^*$. Therefore sub-solutions for the motion
of interface are constructed in \cite{A-Duc} by using the
travelling waves associated with the speeds $c^*-o(\ep)<c^*$.
Since they are changing sign, a slight modification make the
sub-solutions compactly supported from one side. In the slow
exponential decay case we consider, the limit interface moves with
speed $\cl>c^*$. Since the travelling waves with speeds $\cl-o
(\ep)>c^*$ are not changing sign  we are not able to construct
compactly supported sub-solutions.

\vskip 8 pt In the sequel, we denote by $U$ the travelling wave
associated with the speed of the limit interface $\cl=\lambda +
\lambda ^{-1}$ and by $d(t,x)$ the cut-off signed distance
function associated with $\Gamma ^{\cl}$ the solution of the free
boundary problem $(P^{\cl})$. We also consider $V$ the travelling
wave associated with the speed $c_\ep:=\cl-\ep|\ln \ep|$.

 From Lemma \ref{lem:beh-U}, we see that for $U$ (whose speed is
$\cl=\lambda+\lambda^{-1}$),  the $\mu$ that appears in
\eqref{plus-infini} is actually equal to $\lambda$. Therefore, we
collect
\begin{equation}\label{moins-infini-U}
r e^{-\eta |z|} \leq 1-U(z)\leq R e^{-\eta |z|} \quad \text{ for }
z\leq 0\,,
\end{equation}
\begin{equation}\label{plus-infini-U}
r e^{-\lambda  z}\leq U(z)\leq R e^{-\lambda z} \quad \text{ for }
z\geq 0\,.
\end{equation}
Moreover, since $V$ which is slightly slower it will decay faster.
More precisely, we deduce from Lemma \ref{lem:beh-U} that
\begin{equation}\label{moins-infini-V}
r e^{-\eta _\ep  |z|} \leq 1-V(z)\leq R e^{-\eta _\ep |z|} \quad
\text{ for } z\leq 0\,,
\end{equation}
\begin{equation}\label{plus-infini-V}
r e^{-\mu _\ep  z}\leq V(z)\leq R e^{-\mu _\ep z} \quad \text{ for
} z\geq 0\,,
\end{equation}
with $\eta _ \ep \geq \eta + \gamma \ep |\ln \ep|$ and $\mu _ \ep
\geq \lambda + \gamma  \ep |\ln \ep|$, for some $\gamma >0$. The
estimates on the derivatives of $U$ and $V$ corresponding to
\eqref{moins-infini-der} and \eqref{plus-infini-der} also hold.

\vskip 8 pt We are looking for sub-solutions in the form
\begin{equation}\label{sub-sol-motio}
u^-(t,x):=U\left(\frac{d(t,x)+\ep|\ln \ep|m_1e^{m_2
t}}\ep\right)-\ep V\left(\frac{d(t,x)+\ep|\ln \ep|m_1e^{m_2
t}}\ep\right)\,.
\end{equation}
In the sequel we set
\begin{equation}\label{def:z}
z(t,x):=\di \frac{d(t,x)+\ep|\ln \ep|m_1e^{m_2 t}}\ep\,.
\end{equation}

\begin{lem}[Ordering initial data]\label{lem:ordering}
Let Assumptions \ref{H1} and \ref{H3} be satisfied. Then there
exists $\tilde m _1 >0$ such that for all $m_1\geq \tilde m _1$,
all $m_2>0$, all $\ep>0$ small enough, we have
\begin{equation}\label{bla}
u^-(0,x)\leq\ue(t^\ep,x)\,,
\end{equation}
for all $x\in \R ^N$, with $t^\ep=\alpha \ep|\ln\ep|$.
\end{lem}

\begin{proof}Choose $k>0$ so that \eqref{def:k} holds and $m_1\geq
\tilde m_1:=3k$. Note that to prove Lemma \ref{lem:ordering}, it
is sufficient to check that
\begin{equation}\label{goal}
u^-(0,x)=U\left(z(0,x)\right)-\ep V\left(z(0,x)\right)\leq
\underline{u}(t^\ep,x)\,,
\end{equation}
where $\underline u$ is the sub-solution for small times defined
in (\ref{sub-sol}).
To prove this inequality we shall split our arguments into three parts according to the value of $d(0,x)$.\\

If $x$ is such that $d(0,x)\geq 0$. Since $z(0,x)\geq 0$, we
deduce from \eqref{plus-infini-U} that, for $\ep> $ small enough,
$$ u^-(0,x)\leq U(z(0,x))\leq R e^{- \lambda
\frac{d(0,x)}{\ep}}e^{-\lambda m_ 1 |\ln \ep|}\leq \tilde m e^{-
\lambda \frac{|d(0,x)|}{\ep}}\leq \underline u (t^\ep,x)\,.
$$

If $x$ is such that $-k \ep|\ln \ep|\leq d(0,x)\leq 0$. The choice
of $m_1$ implies that  $z(0,x)\geq 2k|\ln \ep|$ so that, for $\ep>
$ small enough,
$$ u^-(0,x)\leq R e^{-\lambda k |\ln\ep|}e^{-\lambda k
|\ln\ep|}\leq \tilde m e^{-\lambda k |\ln\ep|} \leq \tilde m e^{-
\lambda \frac{|d(0,x)|}{\ep}} \leq \underline u (t^\ep,x)\,.
$$

If $x$ is such that $-2d_0 \leq d(0,x)\leq -k\ep|\ln \ep|$, which
implies that $z(0,x)\geq - d_0/\ep$. We claim that (see proof
below), for $\ep>0$ small enough,
\begin{equation}\label{claim}
U'(z)-\ep V'(z)<0\quad\text{ for all } z\in \R\,.
\end{equation}
Therefore, by using \eqref{moins-infini-U} and
\eqref{moins-infini-V} we get
$$
\begin{array}{ll}
u^-(0,x)&\leq (U-\ep V)\left(-d_0/\ep\right)\vsp\\
&\leq 1-re^{-\eta d_0/\ep}-\ep(1-Re^{-\eta_\ep d_0/\ep})\vsp\\
&=1-\ep-e^{-\eta d_0/\ep}(r-\ep R e^{-(\eta_ \ep -
\eta)d_0/\ep})\vsp\\
&\leq 1-\ep-e^{-\eta d_0/\ep}(r-\ep R)\vsp\\
&\leq 1- \ep\,,
\end{array}
$$
for $\ep >0$ small enough. In view of Corollary
\ref{cor:gen-inside}, this completes the proof of \eqref{bla}.

It remains to prove \eqref{claim}. Note that $(U-\ep
V)(-\infty)=1-\ep$ and $(U-\ep V)(\infty)=0$ and assume, by
contradiction, that there exist $\ep_0>0$ and a family
$\{z_\ep\}_{0<\ep <\ep_0}$ such that
$$
U'(z_\ep)-\ep V'(z_\ep)=0 \quad \text{for each } \ep\in
(0,\ep_0)\,.
$$
First assume that $\{z_\ep\}$ is bounded. Then there exists a
sequence $\{\ep_n\}_{n\geq 0}$ tending to zero such that
$z_{\ep_n}\to z_0\in\R$ when $n\to \infty$ for some $z_0$. Recall
that $V$ depends on $\ep>0$ and is uniformly bounded with respect
to $\ep$ up to its second derivative. Passing to the limit $n\to
\infty$  leads us to  $U'(z_0)=0$, a contradiction. Next, assume
that $\{z_\ep\}$ is unbounded. Then there exists a sequence
$\{\ep_n\}_{n\geq 0}$ tending to zero such that $z_{\ep_n}\to
\infty$ or $z_{\ep_n}\to -\infty$ when $n\to\infty$. Consider the
case where $z_ {\ep_n} \to \infty$ then we obtain that
$$
\ep_n=\frac{U'(z_{\ep_n})}{V'(z_{\ep_n})}\geq \frac{re^{-\lambda
z_{\ep_n}}}{R e^{-\mu _ \ep z_{\ep_n}}}\geq \frac r R e^{(\mu _
{\ep_n} - \lambda)z_{\ep_n}} \geq \frac r R\,,
$$
a contradiction with the behavior of $\{\ep_n\}$ when
$n\to\infty$. The case where $z_ {\ep_n} \to -\infty$ is similar.
The claim \eqref{claim} is proved.
\end{proof}

\begin{lem}[Sub-solutions for later times]\label{lem:sub-later}
Recall that $u^-$ was defined in (\ref{sub-sol-motio}) and assume
that $U(0)>\frac{1}{2}$. Then there exists $\tilde m _2
>0$ such that for all $m_1\geq \tilde m _1$, all $m_2\geq \tilde m
_2$, all $\ep>0$ small enough, we have
\begin{equation}
\ep \mathcal L ^\ep [u^- ]=\ep \partial _t u^- -\ep ^2 \Delta u^-
- u^-(1-u^-)\leq 0 \quad \text{ in } (0,\infty)\times \R^N\,.
\end{equation}
\end{lem}

\begin{proof}  By using straightforward computations we get
$$
\begin{array}{lll}
\ep \partial _t u^-=(U'-\ep V')(z)\left(\partial _t
d+m_2\ep|\ln\ep|m_1e^{m_2 t}\right)\vsp\\
\ep ^2\Delta u^- = (U''-\ep V'')(z)|\n d|^2+\ep (U'-\ep V')(z)\Delta d \vsp\\
u^-(1-u^-)=U(1-U)+\ep U V-\ep V(1-U)-\ep ^2 V^2\,,
\end{array}
$$
where $z=z(t,x)$ was defined in \eqref{def:z}. Now, the ordinary
differential equations $U''+\cl U'+U(1-U)=0$ and $V''+(\cl -\ep
|\ln \ep|)V'+V(1-V)=0$ yield
$$
\ep \mathcal L^\ep [u^- ]=E_1+\cdots+E_4\,,
$$
with
\vsp \\
$\qquad\quad  E_1:=(U'-\ep V')(z)\left(\partial _t d+\cl+m_2\pt-\ep\Delta d\right)$\vsp \\
$\qquad\quad  E_2:=\ep ^2|\ln \ep| V'(z)$\vsp \\
$\qquad\quad  E_3:=(U''-\ep V'')(z)(1-|\nabla d|^2)$\vsp \\
$\qquad\quad  E_4:=-\ep V(2U-V)(z)+\ep ^2V^2(z)\,.$\vsp\\
Note that $E_2\leq 0$. We show below that the choice
$$\tilde m
_2:=2A\left(\frac 2{\tilde m _1\eta}+1\right)\,,
$$
is enough to prove the lemma, with $A>0$ the constant that appears
in \eqref{MVT}. To that purpose we distinguish four cases, namely
\eqref{case1}, \eqref{case2}, \eqref{case3} and \eqref{case4}. In
the sequel we denote by $C$ positive constants which do not depend
on $\ep$ (and may change from places to places).

\vsp Assume that
\begin{equation}\label{case1}
-2d_0\leq d(t,x)\leq -\frac{m_2}{2A}\pt\,.
\end{equation}
This implies
\begin{equation}\label{z-infini1}
z\leq -|\ln\ep|m_1\left(\frac {m_2}{2A}-1\right)\leq \frac 2
\eta\ln \ep\,.
\end{equation}
Using the estimates for $U'$ and $V'$ we see that, for $\ep
>0$ small enough,
$$
 |E_1|\leq C(e^{-\eta|z|}+\ep e^{-\eta _\ep |z|})\leq C
 e^{-\eta |z|}\leq C \ep ^2\,,
$$
thanks to \eqref{z-infini1}. Using similar arguments we see that
$|E_3|\leq C\ep^2$. At least note that $z \to -\infty$ as $\ep \to
0$ and that $V(2U-V)(-\infty)=1$. Therefore, if $\ep>0$ is small
enough then $E_4 \leq -\frac 1 2 \ep +C\ep ^2$. It follows that
$\ep \mathcal L^\ep [u^-]\leq - \frac 1 2\ep+C \ep ^2 \leq 0$.

\vsp Assume that
\begin{equation}\label{case2}
 -\frac{m_2}{2A}\pt \leq d(t,x)\leq -\pt\,.
\end{equation}
 From \eqref{norme-un} we deduce that $E_3=0$. From
\eqref{interface}, we deduce that $\partial _t d +\cl=0$. Since,
for $\ep>0$ small enough, $m_2\pt-\ep\Delta d \geq 0$ we deduce
from \eqref{claim} that $E_1\leq 0$. Next, since \eqref{case2}
implies that $z\leq 0$ we get
$$2U(z)-V(z)-\ep V(z) \geq 2 U(0)-1-\ep \geq 0\,,$$
since $U(0)>\frac{1}{2}$. Hence we obtain that $E_4 \leq 0$, so
that we have $\ep \mathcal L^\ep[u^-]\leq 0$.

\vsp Assume that
\begin{equation}\label{case3}
-\pt \leq d(t,x)\leq d_0\,.
\end{equation}
Here, $E_3=0$ and $\partial _t d +\cl=0$ also hold true. Note
that, for $\ep>0$ small enough,
\begin{equation}\label{truc}
m_2\pt-\ep\Delta d \geq \frac{m_1m_2}2 \ep|\ln\ep|\,.
\end{equation}
Using that $U'(z)-\ep V'(z)<\frac 12 U'(z)$ for all $z\in \R$,
whose proof is similar to that of \eqref{claim}, we see that
$$
E_1\leq \frac 12 U'(z)\frac {m_1 m_2}2 \ep |\ln\ep|\leq
-\frac{rm_1m_2}4 \ep |\ln \ep|e^{-\lambda z}\,,
$$
since \eqref{case3} implies $z\geq 0$. Next, we see that $E_4 \leq
C\ep V(z)\leq C\ep e^{-\mu _\ep z}\leq C \ep e^{-\lambda z}$. As a
consequence we get
$$
\ep \mathcal L^\ep[u^-]\leq \ep e^{-\lambda
z}\left(-\frac{rm_1m_2}4|\ln \ep|+C\right)\leq 0\,,
$$
for $\ep >0$ small enough.

\vsp Assume that
\begin{equation}\label{case4}
d_0\leq d(t,x)\leq 2d_0\,.
\end{equation}
We rewrite $ \ep \mathcal L^\ep[u^-]\leq F_1+F_2+F_3$ where
\vsp \\
$\qquad\quad  F_1:=U'(z)\left(\partial _t d+\cl\right)+U''(z)(1-|\n d|^2)$\vsp \\
$\qquad\quad  F_2:=U'(z)(m_2\pt- \ep \Delta d)$\vsp \\
$\qquad\quad  F_3:=-\ep V'(z)(\partial_ t d+\cl+m_2 \pt-\ep \Delta d)-\ep V''(z)(1-|\n d|^2)-\ep V(2U-V)(z)+\ep ^2 V^2(z)\,.$\vsp \\
Since  $|F_3|\leq C\ep e^{-\mu _\ep z}\leq C \ep e^{-\lambda z}$,
we deduce from \eqref{truc} that, for $\ep >0$ small enough,
$$
F_2+F_3\leq \ep e^{-\lambda z}\left(-\frac{rm_1m_2}2|\ln
\ep|+C\right)\leq 0\,.
$$
It remains to estimate the term $F_1$. Since \eqref{case4} implies
that $z\geq d_0/\ep$ we have, for $\ep
>0$ small enough, $U(z)\leq \alpha$ with $0<\alpha<1$ to be selected below. Therefore
it holds that $$U''(z)\leq -\cl U'(z)-(1-\alpha)U(z).$$ Recall
that (see subsection \ref{ss:distance})
$$
\partial _t d=(\partial _
t \tilde d) \zeta'=-\cl \zeta'\,\quad \text{ and }\quad|\n
d|^2=|\n \widetilde d|^2 (\zeta ')^2=(\zeta ')^2\,, $$ where the
function $\zeta '$ is evaluated at point $\widetilde d$. It
follows that
$$
\begin{array}{ll}
F_1&\leq U'(z)(-\cl \zeta ' +\cl)-\cl U'(z)(1-(\zeta ')^2)-(1-\alpha)U(z)(1-(\zeta ')^2)\vsp\\
&\leq -(1-\zeta ')\left[(1-\alpha)(1+\zeta ') U(z)+\cl \zeta '
U'(z)\right]\,.
\end{array}
$$
Since, as $\ep \to 0$, $U'(z)=-\beta e^{-\lambda z}(1+o(1))$ and
$U(z)=\frac \beta \lambda e^{-\lambda z}(1+o(1))$ for some $\beta
>0$, we see that
$$
(1-\alpha)(1+\zeta ') U(z)+\cl \zeta ' U'(z) =
\left[\frac{1-\alpha}\lambda (1+\zeta ')-c\zeta '\right]\beta
e^{-\lambda z}(1+o(1))\,.
$$
Recall that $0\leq \zeta ' \leq 1$ so that by selecting
$\alpha\in(0,\frac{1-\lambda ^2}2)$ we see that
$\frac{1-\alpha}\lambda (1+\zeta ')-c\zeta '\geq 0$ so that, for
$\ep >0$ small enough, $F_1\leq 0$. Hence $\ep \mathcal
L^\ep[u^-]\leq 0$.

The lemma is proved.
\end{proof}

\section{Super-solutions}\label{s:super}

This section is devoted to the construction of super-solutions
which are efficient for the study of both the generation and the
motion of interface.

\begin{lem}[Super-solutions]\label{lem:sup}
Let Assumptions \ref{H1} and \ref{H2} be satisfied. Then there
exists a constant $K_0>1$ such that, for all $\widehat{K}\geq
K_0$, the following holds. For all $x_0\in\Gamma
_0=\partial\Omega_0$,
 for all $\ep >0$ small enough, we have
\begin{equation*}
u^\ep (t,x)\leq \widehat{K} U\left(\frac{(x-x_0)\cdot n_0-\cl
t}{\ep}\right)\quad \text{ for all } (t,x) \in [0,\infty) \times
\R ^N\,,
\end{equation*}
wherein $n_0$ is the outward normal vector to $\Gamma
_0=\partial\Omega_0$ at $x_0$.
\end{lem}

Before proving the lemma, for a given $x\in\R^N$, choose
$x_0\in\partial \Omega_0$ as the projection of $x$ on the convex
$\Omega_0$. For such a choice we have
\begin{equation*}
(x-x_0)\cdot n_0=d(0,x)\,,
\end{equation*}
and the lemma yields, for some $\widehat K >1$,
 \begin{equation}\label{super-sol}
 u^\ep (t,x)\leq \widehat{K} U\left( \frac{d(0,x)-\cl
 t}{\ep}\right)\,,
\end{equation}
for all $t\geq 0$ and all $x\in\R^N$. Next, we prove Lemma
\ref{lem:sup}.

\begin{proof} We recall that $\lambda$ and $M$ were defined in
Assumption \ref{H1} $(iii)$ and that $U(z)\geq r e^{-\lambda z}$
for all $z\geq 0$. Then we define
\begin{equation*}
K_0:= \max\left(1, \frac{M}{r},\;
\frac{\|g\|_\infty+M}{U(0)}\right)\,.
\end{equation*}

Next, let $\widehat{K}\geq K_0$ and $x_0\in\Gamma
_0=\partial\Omega_0$ be given. We consider the map
\begin{equation*}
u ^+(t,x):=\widehat{K} U\left(\frac{(x-x_0)\cdot n_0-\cl
t}{\ep}\right)\,.
\end{equation*}
Straightforward computations and equation ${U} ''+\cl {U} '+U
(1-U)=0$ for the travelling wave yield
\begin{equation*}
\ep \mathcal L^\ep [u ^+](t,x)=\widehat{K}(\widehat{K}-1)(U)^2
\left(\frac{(x-x_0)\cdot n_0-\cl t}{\ep}\right)\,,
\end{equation*}
and therefore $\mathcal L^\ep [u ^+]\geq 0$ in $\mathbb
(0,\infty)\times \R ^N$. Hence, to complete the proof of the lemma
it remains to order the initial data, i.e.
\begin{equation}\label{but}
u_{0,\ep}(x)\leq u ^+(0,x)=\widehat KU\left(\frac{(x-x_0)\cdot
n_0}{\ep}\right)\,,\;\;\forall x\in\R^N\,.
\end{equation}

First, assume that $x$ in the half plane $\{y\in\R^N:
\;\;(y-x_0)\cdot n_0\leq 0\}$. Then since $U$ is decreasing we
have
\begin{equation*}
U\left(\frac{(x-x_0)\cdot n_0}{\ep}\right)\geq U(0)\,.
\end{equation*}
Therefore we obtain that
\begin{equation*}
u_{0,\ep}(x)\leq \frac{\|g\|_\infty+M}{U(0)}
U\left(\frac{(x-x_0)\cdot n_0}{\ep}\right)\leq \widehat KU
\left(\frac{(x-x_0)\cdot n_0}{\ep}\right)\,,
\end{equation*}
in view of  the choice of $K_0$.

Next, we assume that $x$ is in the half plane $\{y\in\R^N:
\;\;(y-x_0)\cdot n_0> 0\}$. Since $\Omega_0$ is convex, we have
$x\notin\Omega_0$ and $d(0,x)\geq (x-x_0)\cdot n_0>0$ so that, in
view of the choice of $K_0$,
$$
\begin{array}{ll}
u_{0,\ep}(x)=h_\ep(x)&\leq  M e^{-\lambda d (0,x)/\ep}\vsp\\
&\leq \widehat K r e^{-\lambda (x-x_0)\cdot n_0/\ep}\vsp\\
&\leq \widehat KU\left(\di\frac{(x-x_0)\cdot n_0}{\ep}\right)\,.
\end{array}
$$

This completes the proof of Lemma \ref{lem:sup}.
\end{proof}

As a consequence of the above construction of super-solutions, we
deduce that, after a very short time, the solution $u^\ep$
approaches 0 in most part of the complementary of the support of
the initial data. The proof is an easy consequence of the upper
bound \eqref{super-sol} combined with the exponential decay
\eqref{plus-infini-U} of $U$. Details are omitted.

\begin{cor}[Generation of interface \lq\lq from the outside"]\label{cor:gen-outside}
Let Assumptions \ref{H1} and \ref{H2} be satisfied.  For any
$p>0$, there exists $k_p>0$ such that, for all $\ep>0$ small
enough,
$$
d(0,x)\geq k_p \ep|\ln \ep|\, \Longrightarrow\, 0\leq
\ue(t^\ep,x)\leq \ep ^p\,,
$$
where we recall that $t^\ep=\alpha \ep|\ln \ep|$.
\end{cor}

\section{Proof of Theorem \ref{th:results}}\label{s:proof}

We first claim that
\begin{equation}\label{recentrage}
 0\leq u^\ep(t+t^\ep,x)\leq 1+\ep\,,
\end{equation}
for all $x\in\R^N$ and all $t\geq 0$. Indeed, the map
\begin{equation*}
\overline{u}(t,x):=w\left(\frac{t}{\ep},
\|g\|_\infty+M\right)\,,\;\;t\geq 0\,,
\end{equation*}
satisfies $\mathcal L^\ep [\overline u]=0$ in $\mathbb
(0,\infty)\times \R ^N$ and $u^\ep (0,x)\leq \|g\|_\infty+M=
\overline{u} (0,x)$ for all $x\in \R ^N$. Therefore the comparison
principle yields
\begin{equation*}
u^\ep (t,x)\leq w\left(\frac{t}{\ep},
\|g\|_\infty+M\right)\,,\;\;\forall t\geq 0\,,\;\forall
x\in\R^N\,.
\end{equation*}
Thus Lemma \ref{lem:w} $(v)$ applies and completes the proof of
\eqref{recentrage}.

\vskip 8 pt Now, let Assumptions \ref{H1}, \ref{H2} and \ref{H3}
be satisfied. Choose $k>0$, $\alpha >0$ as in Corollary
\ref{cor:gen-inside} and $\widehat K >0$ as in Lemma
\ref{lem:sup}. According to Lemma \ref{lem:ordering}, Lemma
\ref{lem:sub-later} and the comparison principle, there exist $m_1
>0$ and $m_2 >0$ such that, for $\ep >0$ small enough,
\begin{equation}\label{dessous}
u^- (t-t^\ep,x)\leq \ue(t,x)\,,
\end{equation}
for all $(t,x)\in[t^\ep,T]\times \R^N$. We choose $\mathcal C$
such that
\begin{equation}\label{choice} \mathcal C>\max\left(\lambda
^{-1},2(\cl \alpha+m_1e^{m_2 T}), 2 \eta ^{-1}\right)\,,
\end{equation}
with $\eta >0$ the constant that appears in
\eqref{moins-infini-U}.

First, we take $x\in (\R^N\setminus \overline{\Omega
^{\cl}_t})\setminus\mathcal N_{\mathcal C\ep|\ln\ep|}(\Gamma
^{\cl} _t)$, i.e.
\begin{equation}\label{d-plus}
d(t,x)\geq \mathcal C \ep|\ln \ep|\,,
\end{equation}
 and prove below that $\ue(t,x)\leq  \ep$, for $t^\ep\leq t\leq
 T$. Since
 $$
 d(t,x)=d(0,x)-\cl t\,,
 $$
 we deduce from  \eqref{super-sol}, the decrease of $U$ and
 \eqref{plus-infini-U} that, for $\ep >0$ small enough, for $0\leq t\leq T$,
 $$
 \ue(t,x) \leq \widehat K U(\mathcal C |\ln \ep|)\leq \widehat K R \ep ^{\lambda \mathcal C}\leq \ep\,,
 $$
 since $\mathcal C >\lambda ^{-1}$.

Next, we take $x\in\Omega ^{\cl} _ t\setminus\mathcal N_{\mathcal
C\ep|\ln\ep|}(\Gamma ^{\cl} _t)$, i.e.
\begin{equation}\label{d-moins}
d(t,x)\leq -\mathcal C \ep|\ln \ep|\,,
\end{equation}
 and prove below that $\ue(t,x)\geq 1-2\ep $, for $t^\ep\leq t\leq
 T$. In view of \eqref{dessous} we have $\ue (t,x)\geq (U-\ep
 V)(z(t-t^\ep,x))$. Note that $d(t-t^\ep,x)=d(t,x)+\cl t^\ep$. Hence the
 choice of $\mathcal C$ implies that $z(t-t^\ep,x)\leq -\frac
 {\mathcal C}2|\ln \ep|$. Therefore, using \eqref{sub-sol-motio} and \eqref{moins-infini-U},
 we see that, for $\ep >0$
small enough,
$$
\begin{array}{lll}
 \ue(t,x) &\geq
U\left(-\frac{\mathcal C}2|\ln\ep|\right)-\ep V\left(-\frac{\mathcal C}2|\ln\ep|\right)\vsp\\
&\geq 1-Re^{-\eta \frac{\mathcal C}2 |\ln\ep|}-\ep\vsp\\
 &\geq 1-2\ep\,,
\end{array}
$$
since $\mathcal C >\di \frac 2 \eta$.

This completes the proof of Theorem \ref{th:results}.

\end{document}